\documentclass[12pt]{article}
\usepackage[colorlinks,allcolors=blue]{hyperref}

\usepackage[utf8]{inputenc}
\usepackage[width=16cm,height=21cm]{geometry}
\usepackage{amsmath,amssymb,amsthm,mathtools,booktabs}
\usepackage{cite}
\usepackage{stackengine}
\stackMath
\usepackage{verbatim} 

\newtheorem{theorem}{Theorem}[section]

\theoremstyle{definition}
\newtheorem{defn}[theorem]{Definition}

\usepackage{xcolor}

\title{
A Nearly Finitary Matroid that is not $k$-Nearly Finitary}
\author{Patrick Tam}

\begin{document}
\maketitle

\begin{abstract}
The class of $k$-nearly finitary matroids for some natural number $k$ is a subclass of the class of nearly finitary matroids.
A natural question is whether this inclusion is proper.
We answer this question affirmatively by constructing a nearly finitary matroid that is not $k$-nearly finitary for any $k \in \mathbb{N}$.
\end{abstract}

\section{Introduction}

We settle a question raised in~\cite{Intersection}. 
Let us first review some background material starting with an axiom system for infinite matroids from~\cite{Axm}.

Now, we introduce a set of matroid axioms developed in~\cite{Axm} in 2010 that captures essential aspects of finite matroid theory while allowing for infinite matroids to be defined.

Now we let $E$ to be any set and possibly infinite. Note that for shorthand, we define $A+b:= A\cup\{b\}$ and $A-b:= A \setminus \{ b \}$.
\begin{defn}
A $\textit{matroid}$ $M$ is a pair $(E,\mathcal{L})$ with $\mathcal{L} \subset 2^{E}$  satisfying the following properties:
\begin{itemize}
	\item I1: $\emptyset \in \mathcal{L}$.
	\item I2: If $B \in \mathcal{L}$ and $A \subset B$, then $A \in \mathcal{L}$.
	\item I3: If $B$ is a maximal element of $\mathcal{L}$ and $A$  is a non-maximal element of $\mathcal{L}$, then there exists $b \in B \setminus A$ such that $A+ b \in \mathcal{L}$.
	\item I4: If $A \in \mathcal{L}$ and $A \subset X \subset E$, then the set $\{S \in \mathcal{L} : A \subset S \subset X\}$ has a maximal element.
\end{itemize}
\end{defn}
Here, we use set inclusion as our partial ordering when we talk about maximality and minimality. We will use this partial order throughout the rest of the paper. Elements of $\mathcal{L}$ are called $\textit{independent sets}$. Maximal elements of $\mathcal{L}$ are also called $\textit{bases}$. These first two axioms are familiar from finite matroids. The third axiom is different from our third axiom for finite matroids because of the possibility of infinite independent sets. It may be possible to extend a countable independent set by another countable independent set under these axioms. A countable set can be a proper subset of another countable set. Thus, cardinality does not give us enough information to determine whether we can extend an independent set by another one. We thus need axiom I3 which is not reliant on comparing cardinalities. The fourth axiom ensures that every matroid $M$ has a base and any restriction of $M$ to any subset of the ground set of $M$ has a base. Elements of $2^E \setminus \mathcal{L}$ are called $\textit{dependent sets}$. A minimal dependent set is called a $\textit{circuit}$. A circuit with only one element is called a $\textit{loop}$. It is possible to define a matroid $M=(E, \mathcal{L})$ by specifying a suitable set of circuits and taking the independent sets to be subsets of $E$ that contain no circuit. Any pair $M=(E, \mathcal{L})$ that satisfies the first two axioms I1 and I2 is called an independence system.

The rank of a matroid is given by the cardinality of a base if this cardinality is finite. If bases of a matroid are infinitely large, we simply say that the rank of that matroid is infinite. In other words, there is only one infinite rank for matroids.

Under the above axioms, every matroid has a base and a dual. Let $M=(E, \mathcal{L})$ be a matroid. We define 
\begin{equation*}
\mathcal{L}^{*} := \{ S \subset E \colon \exists B \in \mathcal{L}^{\mathrm{max}} \text{ s.t. }  S \subset E \setminus  B \}
\end{equation*}
where $\mathcal{L}^{\mathrm{max}}$ is the set of bases of $M$. 
Then $M^{*} = (E, \mathcal{L}^{*})$ is the $\textit{dual}$ of $M$. The authors of~\cite{Axm} who developed this axiom system showed that this dual is indeed a matroid. In fact, this duality is an involution and $M^{**} =M$.

A circuit of $M^*$ is called a $\textit{cocircuit}$ of $M$. A loop of $M^*$ is also known as a $\textit{coloop}$ of $M$. Equivalently, a coloop is an element of $M$ that is not contained in any circuit of $M$ and thus contained in every base of $M$.

A special class of matroids which are known as the finitary matroids are better understood than general infinite matroids. 
\begin{defn}
A matroid $M$ is $\textit{finitary}$ if a set $S$ is independent in $M$ if and only if all finite subsets of $S$ are independent.
\end{defn}
For every matroid $M = (E,\mathcal{L})$, there exists an associated finitary matroid $M^{\mathrm{fin}}=(E,\mathcal{L}^{\mathrm{fin}})$ whose independent sets are subsets $S$ of $E$ such that every finite subset of $S$ is independent in $M$. The proof of this relies on Zorn's lemma and can be found in~\cite{Axm}. $M^{\mathrm{fin}}$ is also known as the $\textit{finitarization}$ of $M$. 

To understand finitarization, consider the matroid $M= ( \mathbb{N}, \mathcal{L})$ where

\begin{equation*}
\mathcal{L} := \{ S \subset \mathbb{N} \colon |\mathbb{N} \setminus S| \geq 2 \}.
\end{equation*}

Since all finite subsets of $\mathbb{N}$ are independent in $M$, $M^\mathrm{fin} = ( \mathbb{N}, 2^{\mathbb{N}})$. Here, $M^\mathrm{fin}$ has more independent sets than $M$. We can make this precise in the following way.

Every base $B$ of $M$ extends to a base $F$ of $M^\mathrm{fin}$. To see why, note that any base $B$ is independent in $M^\mathrm{fin}$. Because of our fourth independence axiom for matroids, the set $\{ S \in \mathcal{L} (M^\mathrm{fin}) \colon B \subset S \subset E \}$ has a maximal element $F$. This maximal element is a base of $M^\mathrm{fin}$. 

Conversely, every base $F$ of $M^\mathrm{fin}$ contains a base $B$ of $M$. To see why, consider $F^* := E \setminus F$. Then $F^*$ is a base of $M^{\mathrm{fin}*}$ . Independent sets of $M^{\mathrm{fin}*}$ are also independent in $M^*$. So $F^*$ extends to some base $B^*$ in $M^*$ by I4. $B := E \setminus B^*$ is then a base in $M$ and $F$ contains $B$.  Unfortunately, the class of finitary matroids is not closed under duality. The authors of~\cite{Union} define the class of nearly finitary matroids which is still not closed under duality but extends the class of finitary matroids.

\begin{defn}
A matroid $M$ is called $\textit{nearly finitary}$ if whenever a base $F$ in $M^\mathrm{fin}$ contains a base $B$ in $M$, their set difference is a finite set.
\end{defn}

\begin{defn}
For an integer $k$, a matroid is called $k$$\textit{-nearly finitary}$ if whenever a base $F$ in $M^\mathrm{fin}$ contains a base $B$ in $M$, their set difference has cardinality bounded by $k$.
\end{defn}

\section{Counterexample}

A natural question raised in \cite{Intersection} is whether every nearly finitary matroid is $k$-nearly finitary for some $k \in \mathbb{N}$. We now present an example of a nearly finitary matroid that is not $k$-nearly finitary for any $k$ that is originally due to Attila Por (personal communication, November 20, 2018).

Before we construct our example, let us first introduce a finitary matroid that will allow us to construct our counterexample.

Note, our set of natural numbers $\mathbb{N}$ excludes zero and consists of positive integers. Let $E$ be the Cartesian product $\mathbb{N} \times \mathbb{N}$. We imagine $E$ to be the integral points of the first quadrant. We define $A\subset E$ to be in $\mathcal{L}_1$ if for every $n\in \mathbb {N}$, the first $n$ rows of $E$ have at most $n$ elements in A. Then $M_1 = (E,\mathcal{L}_1)$ is a finitary matroid. It is straightforward to see that $M_1$ satisfies I1 and I2. For I3, suppose there is a base $B_1$ of $M_1$ and a non-maximal independent set $A_1$. Since $A_1$ is non maximal, there is some base $A_B$ that properly contains $A_1$. Let  $i\in \mathbb{N}$ be a row which contains an element of $A_B \setminus A_1$. Since $B_1$ is maximal, there is some $j \geq i$ such that $B_1$ has exactly $j$ elements in the first $j$ rows. By construction, $A_1$ has strictly less than $j$ elements in the first $j$ rows. There must be some $d \leq j$ such that $A_1$ has strictly less than $n$ elements in the first $n$ rows for all $n \geq d$. We pick the smallest such $d$. $B_1\setminus A_1$ has at least one element $b$ in rows $d$ to $j$. This element can be used to augment $A_1$ and I3 is satisfied. For I4, suppose $A_1 \subset X$ with $A_1 \in \mathcal{L}_1$. Let $A_1[0] = A_1$. We inductively let $A_1 [i]$ be $A_1 [i-1] \cup S[i]$ where $S[i]$ is a maximal subset of row i of $X$ that keeps $A_1 [i]$ independent. Taking the union of all these $A_1[i]$ sets will give us our desired maximal independent set that shows that I4 is satisfied. Finally, to show that $M_1$ is finitary, suppose $A$ is any subset of $E$ whose finite subsets are all independent. If $A$ were dependent, then there would be some $j$ where $A$ has more than $j$ elements in the first $j$ rows. Picking these elements would give us a finite dependent set which contradicts our assumption that all of $A$'s finite subsets are independent. Thus, we conclude that $A$ is independent and that $M_1$ is finitary.

We now define  $\mathcal{L}\subset \mathcal{L}_1 \subset 2^E$. Suppose $A_1$ is independent in $\mathcal{L}_1$. If $A_1$ is finite, then $A_1$ is in $\mathcal{L}$. If $A_1$ is infinite, then we define column $l$ to be $\textit{dominant}$ in $A_1$ if all but finitely many elements of $A_1$ are in column $l$. If the $l$-th column of $A_1$ is dominant, subsets of the form $A_1 \setminus X$ where $X \subset A_1$ and $|X| \geq l$ are elements of $\mathcal{L}$. If there is no such dominant column $l$, then we declare $A_1$ to be an element of $\mathcal{L}$.

\begin{theorem}
$M=(E,\mathcal L)$ as constructed above is a nearly finitary matroid, but not $n$-nearly finitary for any $n\in \mathbb N$.
\end{theorem}

\begin{proof}
$M$ naturally inherits I1 and I2 from $M_1$. It also clearly satisfies I4 because it is nearly finitary. For I3, let $S$ be a non-maximal independent set and let $B$ be a base. $S$ is properly contained in some base $A$. If $A$ and $B$ both have no dominant column, we can extend $S$ with an element of $B$ because $M_1$ is a matroid. If $A$ and $B$ have the same dominant column $k_1$, it is straightforward to see that we can extend $S$ with an element in $B \setminus S$. We follow~\cite{Union} and define the truncation matroid $M_1 [k]$ as follows. $M_1 [k]$ has the same finite independent sets as $M_1$. An infinite set is independent in $M_1 [k]$ if it can be obtained by deleting $k$ elements of an independent set in $M_1$. This construction was proven to be a matroid in~\cite{Union}. We define $M_S (k)$ to be $M_1[k]$ restricted to $A \cup B$. Now, $M_S (k_1)$ is a matroid with independent sets that are independent in $M$ and thus I3 is satisfied. Let us now assume that at least one of $A$ and $B$ have a dominant column but do not share the same dominant column. This is the last remaining case. $B \setminus S$ has elements above the row $m$ for all $m \in \mathbb{N}$. Since $S$ is properly contained in $A$, there is some $a \in A \setminus S$. Let $r$ be the index of the row where this $a$ is. Any element in $B \setminus S$ above row $r$ can be used to augment $S$ to satisfy I3 and finally prove that $M$ is a matroid. It is clear that $M$ is nearly finitary but not $n$-nearly finitary for any $n \in \mathbb{N}$.
\end{proof}

From this example infinite families can be obtained by standard matroid operations defined in~\cite{Union}. Some of these are compiled in the following statements. We omit the (easy) proofs:

\begin{theorem}
Let $M$ be a matroid that is nearly finitary but not $n$-nearly finitary for any $n\in \mathbb N$. Then, the following matroids have the same property:
\begin{enumerate}
\item The truncation $M[k]$ for any $k\in \mathbb N$, as defined in \cite{Union}.
\item The direct sum $M \oplus N$ for any nearly finitary matroid $N$. In other words, this is the matroid union $M \vee N$ where $M$ and $N$ are defined on disjoint ground sets.
\item The deletion $M-S$ for any finite subset $S$ of the ground set.
\item The contraction $M/T$ for any set $T$ of coloops in $M$.
\end{enumerate}
\end{theorem}

%%%%%%%%%%%%%%%%%%%%%%%%%%%%%%%%%%%%%%%%%%%%%%%%%
\subsection*{Acknowledgements}

We thank Attila Por for providing the example described in the paper. Coauthorship was offered to him. We also thank the editor and reviewers for useful feedback to improve this paper.

%BIBLIOGRAPHY
% You do not have to use the same format for your references, but 
%    include everything in this file.  Don't use natbib please.
% If you use BibTeX to create a bibliography, copy the .bbl file into here.
% \newblock is optional (it adds a little space)

\end{document}